\documentclass[12pt]{article}
\usepackage{amsmath}
\usepackage{amsthm}
\usepackage{amscd}
\usepackage{amssymb,amsfonts}
\usepackage[alphabetic,nobysame]{amsrefs}

\usepackage{hyperref}
\usepackage{latexsym}
\usepackage{todonotes}
\usepackage{nicefrac}
\usepackage{epsfig}
\usepackage{stmaryrd}
\usepackage{setspace}
\usepackage{enumerate}
\usepackage[all]{xypic}
\usepackage{bbm,ifpdf,tikz}
\usepackage{verbatim}
\ifpdf
\usepackage{pdfsync}
\fi

\newtheorem{theorem}{Theorem}[section]

\newtheorem{proposition}[theorem]{Proposition}
\newtheorem{corollary}[theorem]{Corollary}

{
\theoremstyle{definition}
\newtheorem{definition}[theorem]{Definition}
\newtheorem{example}[theorem]{Example}

\newtheorem{remark}[theorem]{Remark}
}

\newcommand{\excise}[1]{}

\renewcommand{\dim}{\operatorname{dim}}

\newcommand{\crk}{\operatorname{crk}}

\renewcommand{\and}{\qquad\text{and}\qquad}

\newcommand{\Sn}{\mathfrak{S}}

\newcommand{\Q}{\mathbb{Q}}

\newcommand{\C}{\mathbb{C}}

\newcommand{\IH}{I\! H}

\newcommand{\OS}{\mathcal{O}}

\newcommand{\FS}{\operatorname{FS}}
\newcommand{\FI}{\operatorname{FI}}

\newcommand{\FSop}{\operatorname{FS^{op}}}

\begin{document}
\spacing{1.2}
\noindent{\Large\bf Representation stability in the intrinsic hyperplane arrangements associated to irreducible representations of the symmetric-groups.}\\

\noindent{\bf Ian Flynn, \bf Eric Ramos\footnote{Supported by NSF grant DMS-2137628.}, and \bf Benjamin Young}\\
Department of Mathematical Sciences, Stevens Institute of Technology, Hoboken ,NJ, 07030\\
Department of Mathematics, University of Oregon, Eugene, OR, 97401\\

{\small
\begin{quote}
\noindent {\em Abstract.}
Some of the most classically relevant Hyperplane arrangements are the Braid Arrangements $B_n$ and their associated compliment spaces $\mathcal{F}_n$. In their recent work, Tsilevich, Vershik, and Yuzvinsky \cite{TVY} construct what they refer to as the intrinsic hyperplane arrangement within any irreducible representation of the symmetric group that generalize the classical braid arrangements. Through examples it is also shown that the associated compliment spaces to these intrinsic arrangements display behaviors far removed from $\mathcal{F}_n$. In this work we study the intrinsic hyperplane arrangements of irreducible representations of the symmetric group from the perspective of representation stability. This work is both theoretical, proving representation stability theorems for hyperplane complements, as well as statistical, examining the outputs of a number of simulations designed to enumerate flats.
\end{quote} }

\section{Introduction}

For any $n \geq 1$, the \textbf{braid arrangement} $B_n$ is the hyperplane arrangement within $\mathbb{C}^{n-1} = \mathbb{C}^n/\langle(1,\ldots,1)\rangle$ generated by the hyperplanes $x_i = x_j$, up to simultaneous translation. This particular complex hyperplane arrangement is extremely classical and has been studied from a large number of perspectives.

One particular perspective, that we largely follow in this work, is as follows. Observe that the ambient space $\mathbb{C}^{n-1} = \mathbb{C}^n/\langle(1,\ldots,1)\rangle$ is a representation of the symmetric group $\Sn_n$. In fact, this is the standard representation of $\Sn_n$. The hyperplane $x_i = x_j$ can then be viewed as the fixed locus of this representation with respect to the transposition $(ij)$.

In the work \cite{TVY}, this idea of viewing the braid arrangement in terms of fixed loci of transpositions was extended beyond the standard representation. In particular, for any irreducible representation of the Symmetric group a hyperplane arrangement was constructed that is known as the \textbf{intrinsic hyperplane arrangement} of the irreducible. One may equivalently view these arrangements as the hyperplane arrangements that naturally associate to the corresponding Specht Matroid \cite{WWZ}.

In this paper we will study the intrinsic hyperplane arrangements from the perspective of representation stability in two different styles. Namely, we will consider the cohomology groups of the complementary spaces, as well as their Kazhdan-Lusztig Coefficients \cite{EPW}. While both of these items are known to enjoy a number of stable behaviors \cites{CEF,fs-braid} in the case of the Braid arrangement, we will find that things are not nearly as clear cut in the general circumstance (See Section \ref{notStable}).

\subsection{Representation Stability in the Cohomology groups of the Complementary Spaces}

Given any complex hyperplane arrangement, one of its most significant topological features is its complementary space. In the case of the aforementioned braid arrangement, for instance, the complementary space is the collection of $n$-tuples of non-colliding points in $\C^n$, up to simultaneous translation. This space is more commonly referred to as the \textbf{$n$-pointed configuration space} of $\C^n/\langle(1,\ldots,1)\rangle$.

In his seminal work, Church \cite{Chu} showed that configuration spaces of Manifolds displayed a number of stability phenomena that manifested in their Cohomology groups. This work was then expanded by Church, Ellenberg, and Farb using the language of \textbf{FI-modules} \cite{CEF}.

We write $\FI$ for the category whose objects are the finite sets of the form $[n] := \{1,\ldots,n\}$, and whose morphisms are injective functions of sets. An $\FI$-module is a (covariant) functor from $\FI$ to the category $\mathbb{C}$-vector-spaces. Put more concretely, an $\FI$-module is a collection of complex vector spaces $V := \{V_n\}_{n \in \mathbb{N}}$ such that any injection of sets $[n] \hookrightarrow [m]$ induces a linear map $V_n \rightarrow V_m$, that respects composition in the usual sense. We call these induced maps the \textbf{transition maps} of the module. Observe that injections from a finite set to itself are permutations, and so the vector space $V_n$ is an $\Sn_n$ representation for each $n$. Therefore, an $\FI$-module may be viewed as a collection of $\Sn_n$ representations, with $n$ varying, that are compatible with one another according to the transition maps.

All of the usual constructions from the algebra of modules and rings will have an analog for $\FI$-modules. For instance, a \textbf{submodule} of an $\FI$-module $V$ is an $\FI$-module $W$ such that $W_n \subseteq V_n$ for every $n$, while the transition maps make all necessary diagrams commute. We say that an $\FI$-module $V$ is \textbf{finitely generated}, if $V_n$ is finite dimensional for all $n$ and if there exists an $N \geq 0$ such that for any $n \geq N$ the vector space $V_n$ is spanned by the images of $V_N$ under all transition maps $V_N \rightarrow V_n$. One can think of this condition as asserting that the (necessarily finite) basis elements of $V_N$ span all of the vector spaces $V_n$ for $n \geq N$. The condition of finite generation has a number of very powerful consequences (see Theorem \ref{fgconsequence}).

To state our first stability theorem, we must first introduce some notation. Throughout this work, we will write $\lambda = (\lambda_1,\ldots,\lambda_h)$ to denote a \textbf{partition} of some integer fixed integer $m$. Given a partition $\lambda$ and an integer $n \geq m + \lambda_1$, we write $\lambda[n]$ to denote the \textbf{padded partition} $\lambda[n] = (n-m,\lambda_1,\ldots,\lambda_h)$. In other words, $\lambda[n]$ is obtained from $\lambda$ by adding a large enough first row to create a partition of $n$. It is a fact that the irreducible representations of $\Sn_n$ are in natural bijection with partitions of $n$ (See Remark \ref{Tabloid}), and so we may consider the intrinsic arrangements of the partitions $\lambda[n]$, $\mathcal{A}_n^\lambda$  with $n$ varying. We will be writing $S^\lambda_n$ for the irreducible representation of $\Sn_n$ associated to $\lambda[n]$.

\begin{theorem}\label{CompStab}
    Let $\lambda$ be a fixed partition of some non-negative integer $m$, and let $i \geq 0$ be fixed. Then, writing $\mathcal{M}(\mathcal{A}_n^\lambda)$ for the complementary space of $\mathcal{A}_n^\lambda$, the assignment
    \[
    n \mapsto H^i(\mathcal{M}(\mathcal{A}_n^\lambda))
    \]
    can be extended to a finitely generated $\FI$-module.
\end{theorem}

\begin{remark}
    While we have focused mainly on $\FI$-modules over the complex numbers during this introduction, we will see in the next section that all of these objects can be defined over any Noetherian ring. The above theorem remains true if you take cohomology with integral coefficients.
\end{remark}

Note that the above theorem has two statements within it. Firstly, for any injection of sets $[n] \hookrightarrow [r]$ there is a natural induced map
\[
H^i(\mathcal{M}(\mathcal{A}_n^\lambda)) \rightarrow H^i(\mathcal{M}(\mathcal{A}_r^\lambda)).
\]
Secondly, the $\FI$-module that one forms by taking the vector spaces $H^i(\mathcal{M}(\mathcal{A}_n^\lambda))$ along with the aforementioned induced maps is finitely generated.

If $\lambda = (1)$ then the irreducible representation associated to $\lambda[n]$ is precisely the standard representation. Therefore the above theorem specializes to say that the cohomologies of configuration space of the plane are finitely generated as an $\FI$-module. This has been known since the originating work on $\FI$-modules \cite{CEF}.

We prove Theorem \ref{CompStab} using Gadish's theory of $\FI$-arrangements. This, among other things, tells us that finite generation exists at the level the topological spaces themselves. One consequence of this is that we obtain finite generation results for \'etale cohomology groups \cite{Gad}.

As a consequence of Theorem \ref{CompStab} and Theorem \ref{fgconsequence}, we may say the following:

\begin{corollary}
Let $\lambda$ be a fixed partition of some non-negative integer $m$, and let $i \geq 0$ be fixed. Then,
\begin{enumerate}
    \item There is a finite set $\Lambda \subseteq \bigcup_n \mathbb{Y}_n$, as well as non-negative integers $\{m_\nu\}_{\nu \in \Lambda}$, such that for all $n \gg 0$,
    \[
    H^i(\mathcal{M}(\mathcal{A}_n^\lambda)) \cong \bigoplus_{\nu \in \Lambda} m_\nu S^\nu_n;
    \]
    \item There exists a polynomial $p_V \in \Q[x_1,x_2,\ldots]$, such that for all $n \gg 0$, the character $\chi_n$ of $H^i(\mathcal{M}(\mathcal{A}_n^\lambda))$ satisfies,
    \[
    \chi_n(\sigma) = p_V(X_1(\sigma),X_2(\sigma),\ldots),
    \]
    where $X_i(\sigma)$ is the number of $i$-cycles in the decomposition of $\sigma$. In particular, there is a polynomial $p \in \Q[n]$ such that for all $n \gg 0$, $\dim_\Q H^i(\mathcal{M}(\mathcal{A}_n^\lambda)) = p_V(n)$.
\end{enumerate}
\end{corollary}

\begin{remark}
    It is important to point out that the polynomial $p_V$ is a genuine polynomial, using only finitely many of the variables $x_1,x_2,\ldots$. As a consequence of this, we find that the character of $H^i(\mathcal{M}(\mathcal{A}_n^\lambda))$ does not ``see" cycles above a certain length.
\end{remark}

\subsection{Kazhdan-Lusztig Coefficients}

Given any complex arrangement $\mathcal{A}$ (in fact any matroid), Elias, Proudfoot, and Wakefield define the \textbf{Kazhdan-Lusztig polynomial} of $\mathcal{A}$ \cite{EPW}. This polynomial can either be defined in an entirely combinatorial way using a recursive formula based on the lattice of flats of $\mathcal{A}$, or in a geometric way using the intersection cohomology of the \textbf{reciprocal plane} of the arrangement (see Definition \ref{KLdef}). These polynomials have seen an explosion of study since their inception \cites{fs-braid,thag,kl-survey}.

Relevant for us is the work \cite{fs-braid}. In that paper, Proudfoot and the third author considered the coefficients of the Kazhdan-Lusztig polynomials of the braid arrangements. It was shown that for each $i \geq 0$, the generating function for the $i$-th coefficient was equal to a rational function whose poles are computed. In other words, the $i$-th coefficient is asymptotically exponential in $n$.

In an attempt to generalize \cite{fs-braid}, in this paper we will consider the Kazhdan-Lusztig coefficients of the intrinsic arrangements of the irreducible representations associated to $\lambda[n]$, where $\lambda = (1,\ldots,1) = (1^l)$. Our primary theorem is the following.

\begin{theorem}\label{bestwecando}
Let $\lambda = (1^l)$ for some $l \geq 1$ and write $\mathcal{A}_n$ for the intrinsic arrangement of $\lambda[n]$. Then there exists a function $f_l(n):\mathbb{N} \rightarrow \mathbb{N}$ such that for all $n \geq 0$, $f_l(n)$ is a lower bound on the first Kazhdan-Lusztig coefficient of $\mathcal{A}_n$, and
\[
\lim_{n \rightarrow \infty}f_l(n)(l+1)!/(l+1)^n = 1.
\]
\end{theorem}

It was proven in \cite{EPW} that the first Kazhdan-Lusztig coefficient of the braid arrangement $\mathcal{A}_n^1$ is equal to $2^{n-1}-1-\binom{n}{2}$. The above theorem tells us that for $l > 1$ this coefficient will be bounded from below by something which is asymptotically like $(l+1)^n$. In fact, we prove something stronger, which is that all character behaviors of the \emph{equivariant} Kazhdan-Lusztig coefficients in this case are bounded from below by uniformly described exponentials (see Remark \ref{expansion}).

Of course, one hopes that Theorem \ref{bestwecando} can be improved to show this inequality is asymptotically sharp. While we cannot prove that this is not the case, we discuss in Section \ref{notStable} that it seems quite unlikely, even in the simplest case $l=2$. To briefly discuss why this is so, we note that \cite{fs-braid} prove their asymptotic statements using what are known as \textbf{$\FSop$-modules} (see Section \ref{fsmod}). While we show that this module structure still exists for $l > 1$, in Section \ref{notStable} we use computer algebra and statistical methods illustrate that the structure likely lacks the necessary finite generation properties. This is also why we are currently unable to prove much about the other coefficients of the Kazhdan-Lusztig polynomial in these cases. While the methods of this paper can likely be used to obtain exponential lower bounds on these numbers, it is currently out of reach to say anything else.

\section*{Acknowledgements}
ER Was supported by NSF grant DMS-2137628. He is grateful to Nick Proudfoot and Ben Leinwand for a number of helpful discussions.

\section{Background}

\subsection{The representation theory of the symmetric group and representation stability}.

In this section we briefly review the necessary background from the (complex) representation theory of the symmetric group that will be needed in this work. We then dive into the theory of representation stability and, in particular, the theory of FI-modules that will provide the technical hammer needed to prove the main results of the work. All of the theorem statements about the complex representation theory of the symmetric group can be found in any standard text on the subject, for instance \cite{CSST}.

Throughout this work we will write $\Sn_n$ to denote the symmetric group on $n$-letters. 

\begin{definition}
Let $n \geq 0$ be fixed. A \textbf{partition} of $n$ is a tuple $\lambda = (\lambda_1,\ldots,\lambda_r)$ of positive integers such that $\sum_i \lambda_i = n$, and $\lambda_i \geq \lambda_{i+1}$ for each $i$. By convention, when $n = 0$, we say that there is precisely one partition of $n$, namely the empty partition. We write $\mathbb{Y}_n$ to denote the set of all partitions of the integer $n$.

Relatedly, a \textbf{partition of $[n]$} is a collection of non-empty, disjoint subsets of the set $[n] = \{1,\ldots,n\}$ whose union is all of $n$. If $\lambda$ is some partition of $n$, we say that a partition of $[n]$, $\alpha$, has \textbf{shape $\lambda$} if the sizes of the subsets present in $\alpha$ can be organized to form $\lambda$. We write $\mathbb{P}_n$ to denote the set of all partitions of $[n]$, and $\mathbb{P}_n(\lambda)$ to denote the subset of partitions with shape $\lambda$.
\end{definition}

It is a well known fact that the irreducible complex representations of $\mathbb{S}_n$ are in bijection with partitions of $n$. We will generally write $S^\lambda$ to denote the irreducible representation associated to the partition $\lambda$.

\textbf{It is our convention that the irreducible corresponding to the partition $(n)$ is the trivial representation.}

\begin{remark}
It is common to visualize partitions of $n$ as tableau of boxes, where the $i$-th row of boxes has $\lambda_i$ boxes. This visualization is referred to as the \textbf{Young diagram} associated to $\lambda$. We will frequently jump between this visualization and the ``tuple-form" of a partition for certain definitions. For instance, in Definition \ref{intrinsic}, we define one partition to be the transpose of another.
\end{remark}

\begin{remark}\label{Tabloid}
While it is not absolutely critical to know precisely how the irreducible representation partition correspondence works, there is one particular presentation of these modules that we will continuous refer to, and so we take a moment to define it now. Fix a partition $\lambda$ of $n$. Consider the set $\mathcal{T}$ of all bijections $T$ from the Young diagram of $\lambda$ to $[n]$. That is to say, $T$ assigns in a bijective fashion numbers $\{1,2,\ldots,n\}$ to each box of the Young diagram. If $R(T) \subseteq \Sn_n$ is the subgroup of permutations which preserve the elements in the rows of $T$, then we write $\{T\}$ to denote the orbit of $T$ under the action of $R(T)$. The object $\{T\}$ is often referred to as a \textbf{tabloid} with shape $\lambda$.

The set of all tabloids of shape $\lambda$ form an $\Sn_n$-set in the obvious way, and we may write $M^\lambda$ to denote the $\Sn_n$-representation which linearizes this set. The irreducible representation $S^{\lambda}$ can be found as a subrepresentation of $M^\lambda$ as follows. For each bijection $T$ from the diagram of $\lambda$ to $[n]$, define
\[
v_T = \sum_{\sigma \in C(T)}\text{sgn}(\sigma)\{\sigma(T)\},
\]
where $C(T)$ is the subgroup of permutations that perserve the elements in the columns of $T$. The irreducible representation $S^\lambda$ is the span of the vectors $v_T$.
\end{remark}

\begin{example}
The irreducible representation $S^{(n-1,1)}$ corresponding to the partition $(n-1,1)$ is the standard representation. That is to say, the $n-1$ dimensional subrepresentation of $\C^n$ spanned by the differences of the canonical basis vectors. More generally, let $k \geq 0$ be fixed. Then the $\Sn_n$ representation associated to the hook partition $(n-k,k)$ is the exterior power $\bigwedge^k S^{(n-1,1)}$.
\end{example}

The following theorem will arise later for technical reasons.

\begin{theorem}\label{selfdual}
Let $n \geq 0$ be an integer, and let $\lambda$ be a partition of $n$. Then $\Sn_n$ acts on the vector space $S^\lambda$ by matrices with rational entries. In particular, the dual of any complex symmetric group representation is isomorphic to the original representation.
\end{theorem}

The hook partitions $(n-k,k)$ are one particular example of a combinatorial object that will be very import for us going forward.

\begin{definition}
Let $\lambda$ be a fixed partition of some integer $m$, and let $n \geq m + \lambda_1$. Then we define the \textbf{padded partition}
\[
\lambda[n] := (n-|\lambda|,\lambda)
\]

We will typically write $S^\lambda_n$ for the irreducible representation corresponding to $\lambda[n]$. Moreover, we will implicitly assume that $S^\lambda_n = 0$ for any $n < |\lambda| + \lambda_1$.
\end{definition}

The representations $S^\lambda_n$ form the basis for what we call representation stability. Roughly speaking, one thinks of the partitions $S^\lambda_n$, with varying $n$, as being the ``same" representation manifested across different symmetric groups. Therefore, to see that a sequence of $\Sn_n$-representations, with $n$ varying, is ultimately expressible as a direct sum of irreducible representations of padded partitions with constant multiplicity and allowable $\lambda$, is a form of stability. Our next goal is to make this all precise.

\begin{definition}
Let $\FI$ denote the category whose objects are the sets $[n] = \{1,\ldots, n\}$, and whose morphisms are injections of sets. An \textbf{$\FI$-module over a ring $k$} is a functor,
\[
V: \FI \rightarrow k-\text{mod},
\]
from $\FI$ to the category of finitely generated $k$-modules. We will usually write $V_n := V([n])$ and $f_\ast := V(f)$. 

We say that an $\FI$-module $V$ is \textbf{finitely generated} if there is an integer $N$ such that for all $n \geq N$, the group $V_n$ is generated as an $\Sn_n$-module by the image of $V_N$ under the map induced by the standard inclusion $[N] \hookrightarrow [n]$ In this case we also say that $V$ is \textbf{generated in degrees $\leq N$}.
\end{definition}

Implicit in the above definitions is the fact that, for an $\FI$-module $V$, the group $V_n$ is always an $\Sn_n$-module. Indeed, this simply derives from the fact that the endomorphisms of an object in $\FI$ are precisely the permutations. One therefore arrives at the following ``concrete" description of an $\FI$-module: An $\FI$-module is a sequence of finitely generated $k[\Sn_n]$-modules, one for each $n \geq 0$, that are compatible with one another in accordance with maps induced by inclusions $[m] \hookrightarrow [n]$. The additional assumption of being finitely generated then imposes the condition that, whenever $n$ is sufficiently large, these compatibility maps become isomorphisms, up to the action of the symmetric group.

The following theorem collects some relevant conclusions that one draws from the fact that a given $\FI$-module is finitely generated.

\begin{theorem}[\cite{CEF}]\label{fgconsequence}
Let $V$ be a finitely generated $\FI$-module over a field $k$ of characteristic 0 that is generated in degrees $\leq N$. Then,
\begin{enumerate}
    \item There is a finite set $\Lambda \subseteq \bigcup_n \mathbb{Y}_n$, as well as non-negative integers $\{m_\lambda\}_{\lambda \in \Lambda}$, such that for all $n \gg 0$,
    \[
    V_n \cong \bigoplus_{\lambda \in \Lambda} m_\lambda S^\lambda_n;
    \]
    \item The set $\Lambda$ from the above part consists exclusively of partitions with at most $N$ parts;
    \item There exists a polynomial $p_V \in \Q[x_1,x_2,\ldots]$ of degree $\leq N$, such that for all $n \gg 0$, the character $\chi_n$ of $V_n$ satisfies,
    \[
    \chi_n(\sigma) = p_V(X_1(\sigma),X_2(\sigma),\ldots),
    \]
    where $X_i(\sigma)$ is the number of $i$-cycles in the decomposition of $\sigma$. In particular, there is a polynomial $p \in \Q[n]$ such that for all $n \gg 0$, $\dim_\Q V_n = p_V(n)$.
    
\end{enumerate}
\end{theorem}

\begin{remark}
There has also been considerable work done in consequences of finite generation on the modular side. While we do not consider the modular situation in this work, one can see \cites{Kriz,NS} for more information.
\end{remark}

In later sections we will consider other representation stability type structures such as $\FI$-arrangements (Definition \ref{fiarr}), and $\FS^{op}$-modules (Section \ref{fsmod}). Each of these topics will be covered as needed.

\subsection{Hyperplane arrangements and their Kazhdan-Lusztig polynomials}

To begin this section we review the construction of \cite{TVY}, of the intrinsic hyperplane arrangement associated with the representation $S^\lambda$.

\textbf{For the remainder of this section we fix a partition $\lambda$ of some non-negative integer $n$.}

\begin{definition}\label{intrinsic}
Let $\lambda^\dagger$ denote the \textbf{conjugate} partition of $\lambda$, obtained from $\lambda$ by transposing its Young diagram. If $\alpha \in \mathbb{P}_n(\lambda^\dagger)$ and $\alpha $ has parts $[n] = \cup_i \alpha_i$, we define $\Sn_\alpha$ to be the subgroup $\prod_i \Sn_{\alpha_i}$ of $\Sn_n$. We will also write $T_\alpha$ to denote a set of transpositions that generate the group $S_\alpha$.

With notation as above, we may now define the \textbf{intrinsic hyperplane arrangement of $S^\lambda$} to be that comprised of the hyperplanes
\[
H_{\alpha} := \sum_{\tau \in T_\alpha} \left( S^\lambda \right)^\tau,
\]
where $\left(S^\lambda\right)^\tau$ invariant subspace of $S^\lambda$ with respect to the transposition $\tau$. We write $\mathcal{A}_\lambda$ for this hyperplane arrangement.
\end{definition}

It is entirely non-obvious that the above definition is well defined. However, it is shown in \cite{TVY} that the spaces $H_{\alpha}$ are hyperplanes, and that they do not depend on the choice of generating transpositions $T_\alpha$.

One should observe that for any permutation $\sigma \in \Sn_n$, the action of $\sigma$ on $S^{\lambda}$ descends to an action of $\sigma$ on $\mathcal{A}_\lambda$. Indeed, it isn't hard to see that
\[
\sigma \cdot H_{\alpha} = H_{\sigma\cdot \alpha}
\]
for any $\alpha \in \mathbb{P}_n(\lambda^\dagger)$. Indeed, this just follows from the fact that for any transposition $\tau$, $v \in S^\lambda$ is fixed by $\tau$ if and only if $\sigma \cdot v$ is fixed by $\sigma \tau \sigma^{-1}$. Later we will see that this action on the hyperplane arrangement can be lifted to an action by the category $\FI$.

\begin{example}
Let $\lambda = (n-1,1)$. In this case we have seen that $S^\lambda$ is the standard representation. Moreover, we have that $\lambda^\dagger = (2,1^{n-k})$. Therefore any $\alpha \in \mathbb{P}_n(\lambda^\dagger)$ is uniquely determined by a subset of $[n]$ of size two. We will abuse notation here and just think of $\alpha$ as being this subset.

If $\alpha = \{i,j\}$, then $T_\alpha = \{(ij)\}$ and $H_\alpha$ is the subspace of $S^\lambda$ obtained from the hyperplane $x_i = x_j$ in $\C^n$ by orthogonal projection onto $S^\lambda$.

Of course, the hyperplane arrangement $x_i = x_j$ in $\C^n$ is an extremely important arrangement. We therefore can think of $\mathcal{A}_\lambda$ as this arrangement with the (non-trivial) intersection of all the hyperplanes removed. That is to say, the braid arrangement. In \cite{TVY}, the arrangement $\mathcal{A}_\lambda$ is explored more deeply in the cases of other hook partitions. In fact, it is argued that for $k \geq 2$, the hook partition $(n-k,1^k)$ corresponds to a hyperplane arrangement $\mathcal{A}_\lambda$ whose complimentary space is \emph{not} $K(\pi,1)$. This lands in pretty stark contrast to the situation of the braid arrangement, which was famously proven to be $K(\pi,1)$ by Arnol'd \cite{Ar}. 
\end{example}

Moving away from the specific case of the intrinsic hyperplane arrangement, we conclude this section by briefly touching on the theory of Kazhdan-Lusztig polynomials of complex hyperplane arrangements. These objects were defined at the level of general matroids in \cite{EPW}, where it is shown that in the realizable setting (i.e. the hyperplane arrangement setting) they could be interpreted purely geometrically. This is the approach we consider in this work. Following this pioneering work, \cite{GPY} then considered cases wherein there was a natural group action on the matroid or hyperplane arrangement. This equivariant version of the theory has shown itself to be very naturally amenable to tools in representation stability and related machinery \cite{fs-braid, PR-trees,PR-genus}. We continue this trend here.

\begin{definition}\label{KLdef}
Let $(V,\mathcal{A})$ denote a complex hyperplane arrangement \footnote{In this work all hyperplane arrangements are assumed to be central, and finite.} and for each $H \in \mathcal{A}$ choose an orthogonal vector $a(H)$. The \textbf{complement} of $\mathcal{A}$ is then defined to be the variety
\[
\mathcal{M}(\mathcal{A}) = \{v \in V \mid \langle v,a(H) \rangle \neq 0 \text{ for all $H \in \mathcal{A}$}\}
\]
The variety $\mathcal{M}(\mathcal{A})$ can be embedded into the torus $(\C^\times)^{\mathcal{A}}$ by setting the $H$-th coordinate to be $\langle a(H), \bullet \rangle$. We define $\mathcal{M}^{-1}(\mathcal{A})$ to be the image of $\mathcal{M}(\mathcal{A})$ under the involution of the torus which inverts every coordinate. Finally, the \textbf{reciprocal plane} $X_{\mathcal{A}}$ is then defined to be the closure of $\mathcal{M}^{-1}(\mathcal{A})$ in the affine space $\C^{\mathcal{A}}$.

The Kazhdan-Lusztig polynomial of the hyperplane arrangement $\mathcal{A}$, $p_{\mathcal{A}}(t)$, is defined to be the Poincar\' e polynomial associated to the \emph{intersection cohomology}. That is to say, the coefficient in front of $t^i$ in $p_{\mathcal{A}}(t)$ is precisely given by $\dim_\C IH^{2i}(X_{\mathcal{A}};\C)$.

If the arrangement $\mathcal{A}$ is equipped with an action by some finite group $\Gamma$, then this action descends to an action on intersection homology. In this case we may define the \textbf{equivariant Kazhdan-Lusztig polynomial} $p_\mathcal{A}^\Gamma (t)$ to be the formal polynomial whose coefficient of $t^i$ is the virtual representation associated with $IH^{2i}(X_{\mathcal{A}};\C)$.
\end{definition}

\begin{example}
We note that if $\Gamma = \{1\}$ is the trivial group, then the two kinds of Kazhdan-Lusztig polynomials coincide. For a non-trivial example, let $\mathcal{A} = \{H_{i,j}\}$ denote the intrinsic arrangement of $S^{(n,1)}$. In this case, both the equivariant (with respect the action of $\Sn_n$) and standard Kazhdan-Lusztig polynomials are fairly well understood both theoretically \cite{FL}, as well as computationally \cite{GPY,fs-braid}. Using techniques coming from the theory of representation stability, it was shown in \cite{fs-braid} that the $i$-th coefficient of the Kazhdan-Lusztig polynomial of $\mathcal{A}$ must grow like a particular kind of exponential in $n$. On the equivariant side, it was also shown that the irreducible representations that appear as coefficients have restrictions in the shapes of their associated partitions.

Our primary theorem on Kazhdan-Lusztig polynomials of intrinsic arrangements (Theorem \ref{bestwecando} will partially generalize the main theorem of \cite{fs-braid} to all intrinsic hyperplane arrangements of hooks. Through explicit computations, however, we will see that the full strength of the theorems in \cite{fs-braid} do not continue to be true beyond the case of the braid arrangement (Section \ref{notStable}).
\end{example}

\section{The complements of intrinsic arrangements}\label{sec:Comp}

In this section we prove our first representation stability theorem \ref{CompStab}. To begin our proof we must first recall some notation from Gadish's work on $\FI$-arrangements \cite{Gad}.

\begin{definition}\label{fiarr}
Let $(V,\mathcal{A})$ and $(W,\mathcal{B})$ be linear hyperplane arrangements over $\C$. A \textbf{morphism} between them $\varphi:(V,\mathcal{A}) \rightarrow (W,\mathcal{B})$ is a surjection of vector spaces $\varphi^\ast:W \rightarrow V$ such that for any $H \in \mathcal{B}$, one has that $\varphi^{-1}(H) \in \mathcal{A}$. In this way, a morphism of arrangements induces both:
\begin{itemize}
    \item A regular map $\mathcal{M}(\mathcal{B}) \rightarrow \mathcal{M}(\mathcal{A})$ and;
    \item A map between the semi-lattices induced by the intersections of the hyperplanes defining $\mathcal{A}$ and $\mathcal{B}$, respectively. We write $\mathcal{L}(\mathcal{A})$ to denote the intersection semi-lattice of $\mathcal{A}$. We usually will refer to this as the \textbf{lattice of flats} of $\mathcal{A}$
\end{itemize}

An $\FI$-arrangement is a functor $(V_\bullet,\mathcal{A}_\bullet)$ from $\FI$ to the category of arrangements. We say that an $\FI$-arrangement $(V_\bullet,\mathcal{A}_\bullet)$ is \textbf{finitely generated} if there is an integer $N$ such that for all $n \geq N$ and every subspace $S \in \mathcal{L}(\mathcal{A}_n)$, $S$ can be written as an intersection of preimages of subspaces coming from $\mathcal{L}(\mathcal{A}_N)$. 

Finally, an $\FI$-arrangement is \textbf{normal} if for every morphism $f:[m] \hookrightarrow [n]$ in $\FI$ and every subspace $S \in \mathcal{L}(\mathcal{A}_n)$ that contains the kernel of the map induced by $f$ on vector spaces, $S$ is in the image of the map induced by $f$ on the semi-lattice of intersections
\end{definition}

\begin{example}

Consider the $\FI$-arrangement $(V_\bullet,\mathcal{A}_\bullet)$, where $V_n = \C^n$ and $\mathcal{A}_n$ is the collection of hyperplanes defined by $x_i = x_j$ for each $i \neq j$. Note in this case the maps being induced by the morphisms of $\FI$ are the forgetful maps on $\C^n$.

We claim that this $\FI$-arrangement is both finitely generated and normal. Indeed, for any $n \geq 2$ every subspace $S \in \mathcal{L}(\mathcal{A}_n)$ is in natural bijection with a set partition of $[n]$, i.e. the subspace which equates those coordinates that are in the same part. This, however, is clearly just the intersection of the preimages of $x_1 = x_2$ in $\mathcal{A}_2$ along the various forgetful maps. For normality, let $f:[m] \hookrightarrow [n]$ be an injection of sets. The kernel of $f^\ast:\C^n \rightarrow \C^m$ is comprised of precisely those vectors that are zero in the coordinates corresponding to the image of $f$. Letting the image of $f$ be the set $\{i_1,\ldots,i_m\}$, we immediately see that the subspaces $S \in \mathcal{L}(\mathcal{A}_n)$ that contain this kernel are those that only equate the coordinates corresponding to some subset of the indices in $\{i_1,\ldots,i_m\}$. However, any such subspace is clearly in the image of the map induced by $f$ on the intersection lattice.

\end{example}

Our primary interest in this machinery is the following theorem, proven by Gadish in \cite{Gad}.

\begin{theorem}[Theorem A, \cite{Gad}]\label{GadThm}
If $(V_\bullet,\mathcal{A}_\bullet)$ is a finitely generated and normal $\FI$-arrangement, then for any $i \geq 0$ the rational cohomology of the complement space,
\[
H^i(\mathcal{M}(\mathcal{A}_\bullet);\Q)
\]
forms a finitely generated $\FI$-module over $\Q$.
\end{theorem}

Our way forward has therefore been made clear. We must show,
\begin{itemize}
    \item For any partition $\lambda$, the collection of hyperplane arrangements $(S^{\lambda}_n,\mathcal{A}_{\lambda[n]})$ can be glued together in some natural way to form an $\FI$-arrangement and,
    \item that this $\FI$-arrangement is both finitely generated and normal.
\end{itemize}

We take on these two problems in turn.

To begin, fix a partition $\lambda$ and recall from \cite{CEF} that there exists a (finitely generated) $\FI$-module over $\Q$ whose evaluation at $[n]$ is the module $S^{\lambda}_n$ so long as $n \geq \lambda_1 + |\lambda|$ and the zero module otherwise. Denote this $\FI$ module by $S^{\lambda}_\bullet$. By dualizing, one obtains an $\FI^{op}$-module whose evaluation at $[n]$ is the module
\[
(S^{\lambda}_n)^\vee \cong S^{\lambda}_n,
\]
by Theorem \ref{selfdual}. In fact, for any injection $f:[m] \hookrightarrow [n]$, the maps induced by $f$ on the original $\FI$-module, $f_\ast$, and its dual, $f^\ast$, are partial inverses of one another. This grants us that the map induced by $f:[m] \rightarrow [n]$,
\[
f^\ast:S^{\lambda}_n \rightarrow S^{\lambda}_m
\]
is surjective, and that for any $\alpha \in \mathbb{P}_m(\lambda[m]^\dagger)$
\[
(f^{\ast})^{-1}(H_{\alpha}) = H_{f(\alpha)} \in \mathbb{P}_n(\lambda[n]^\dagger)
\]
In particular, we have an $\FI$-arrangement whose evaluation at each $n$ is precisely the intrinsic hyperplane arrangement,
\[
(S^{\lambda}_n,\mathcal{A}_{\lambda[n]}),
\]
as desired. To prove our representation stability theorem it therefore remains to show that this $\FI$-arrangement is both finitely generated and normal.

\begin{proposition} \label{fgnorm}
The $\FI$-arrangement $(S^{\lambda}_\bullet,\mathcal{A}_{\lambda[\bullet]})$ is both finitely generated and normal.
\end{proposition}

\begin{proof}
To begin, we will continue to use the convention that a set partition $\alpha \in \mathbb{P}_n(\lambda[n]^\dagger)$ is identified with the pieces of the partition of size strictly greater than 1. In this way $\mathbb{P}_n(\lambda[n]^\dagger)$ becomes an $\Sn_n$ set by the obvious action. In fact, one may think of $\mathbb{P}_n(\lambda[n]^\dagger)$ as being a (finitely generated) $\FI$-set in the sense of \cite{RSW}. In particular, for any $n \gg 0$ and any finite collection of elements in $\mathbb{P}_n(\lambda[n]^\dagger)$, $\{\alpha_i\}_i$, one may find injections $f_i:[n-1] \hookrightarrow [n]$ and partitions $\{\beta_i\}_i$ such that for each $i$, $\alpha_i = f_i(\beta_i)$. Finite generation now follows from this as well as the fact that the action of $\FI$ on $\mathcal{A}_{\lambda[\bullet]}$ agrees with its action on $\mathbb{P}_\bullet(\lambda[\bullet]^\dagger)$.

It therefore remains to show that these arrangements are normal. Let $f:[m] \hookrightarrow [n]$ be an injection of sets. We first hope to understand which subspaces in $\mathcal{L}(\mathcal{A}_{\lambda[n]})$ will contain the kernel of the map induced by $f$ on the underlying vector spaces. More specifically, we only need to understand when a given hyperplane $H_{\alpha}$ contains the kernel of this map, for $\alpha \in \mathbb{P}_n(\lambda[n]^\dagger)$. Tracing through definitions we find that $H_{\alpha}$ will contain the kernel of the map induced by $f$ whenever the parts of $\alpha$ only contain elements from the image of $f$. Now given some intersection of hyperplanes with this condition on $\alpha$, We find that it will be in the image of the map induced by $f$ on the intersection semi-lattice by looking at the intersection of hyperplanes whose associated set partition contains the relevant preimages (under $f$) of those image values that appeared in the original hyperplanes' associated partition.
\end{proof}

\begin{remark}
In the original cited work \cite{Gad}, the author is able to conclude the stronger condition of $H^i(\mathcal{M}(\mathcal{A}_\bullet);\Q)$ being free as an $\FI$-module, provided that the $\FI$-arrangement satisfies a condition he calls \textbf{continuity}. Our primary examples of interest, namely the intrinsic hyperplane arrangements, are unfortunately not always continuous! Indeed, consider the case where $\lambda = (1,1,1)$. Then the following pullback diagram in $\FI$,
\[
\begin{CD}
\emptyset @>>> [4]\\
@VVV           @V (12) VV\\
[4]       @>(34)>> [4]
\end{CD}
\]
is mapped to the following diagram by $S^{\lambda}_\bullet$,
\[
\begin{CD}
0 @<<< S^{(1,1,1,1)}\\
@AAA           @A -1 AA\\
S^{(1,1,1,1)}       @<-1<< S^{(1,1,1,1)}
\end{CD}
\]
This diagram is not a pushout diagram in the category of vector spaces, thus we must conclude that the associated $\FI$-arrangement is not continuous.
\end{remark}

Proposition \ref{fgnorm} is all we need to conclude the main theorem of this section.

\begin{theorem}\label{compfg}
Fix a partition $\lambda$ and $i \geq 0$, and write $V = H^i(\mathcal{M}(\mathcal{A}_{\lambda[\bullet]});\Q)$. Then,
\begin{enumerate}
    \item There is a finite set $\Lambda \subseteq \bigcup_n \mathbb{Y}_n$, as well as non-negative integers $\{m_\mu\}_{\mu \in \Lambda}$, such that for all $n \gg 0$,
    \[
    V_n \cong \bigoplus_{\mu \in \Lambda} m_\mu S^\mu_n;
    \]
    \item The set $\Lambda$ from the above part consists exclusively of partitions with at most $i*(|\lambda| + \lambda_1)$ parts;
    \item There exists a polynomial $p_V \in \Q[x_1,x_2,\ldots]$ of degree $\leq i\cdot(|\lambda| + \lambda_1)$, such that for all $n \gg 0$, the character $\chi_n$ of $V_n$ satisfies,
    \[
    \chi_n(\sigma) = p_V(X_1(\sigma),X_2(\sigma),\ldots),
    \]
    where $X_i(\sigma)$ is the number of $i$-cycles in the decomposition of $\sigma$. In particular, there is a polynomial $p \in \Q[n]$ such that for all $n \gg 0$, $\dim_\Q V_n = p_V(n)$.
\end{enumerate}
\end{theorem}

\begin{proof}
In view of Proposition \ref{fgnorm}, as well as Theorems \ref{GadThm} and \ref{fgconsequence}, the only part of this theorem that remains unjustified are the bounds of $i\cdot(|\lambda| + \lambda_1)$ that appear throughout. We will accomplish this using a few key facts from the study of hyperplane arrangements. Firstly, the famous theorem of Orlik and Solomon \cite{OS} tells us that for any fixed $n \geq |\lambda| + \lambda_1$,
\begin{itemize}
    \item The first cohomology group $H^1(\mathcal{M}(\mathcal{A}_{\lambda[n]});\Q)$ is a vector space with basis in bijection with the hyperplanes of $\mathcal{A}_{\lambda[n]}$, and\\
    \item for any $i > 1$, the cohomology group $H^i(\mathcal{M}(\mathcal{A}_{\lambda[n]});\Q)$ is a quotient of the tensor power $(H^1(\mathcal{M}(\mathcal{A}_{\lambda[n]});\Q))^{\otimes i}$
\end{itemize}
The first fact can be applied to show that, as an $\FI$-module, $H^1(\mathcal{M}(\mathcal{A}_{\lambda[\bullet]});\Q)$ is generated in degree $|\lambda| + \lambda_1$. The bounds of the original Theorem \ref{fgconsequence} now imply our bounds for this module. For higher $i$, one simply applies the fact that generating degree is additive in tenor powers of $\FI$-modules \cite{CEF}, and the second fact found above. Note that while this is not enough to prove that the higher cohomologies are generated in degrees at most $i\cdot(|\lambda|+\lambda_1)$, it certainly is strong enough to conclude the upper bounds we are looking for, as upper bounds on the irreducible decomposition and dimension growth will descend to quotients.
\end{proof}

\begin{remark}
One might observe that the above proof can be enhanced so-as to entirely avoid the theory of $\FI$-arrangements. Indeed, it would essentially amount to showing that the quotient maps from the tensor powers of the first cohomology to the $i$-th cohomology respected all of the $\FI$-structure. Such a proof, however, would essentially only give us statements about these cohomology groups at the level of algebra. In contrast, our method can be thought of as proving finite generation at the level of the topological spaces themselves. This has utility in-so-far-as applications to \'etale cohomology, as detailed in \cite{Gad}.
\end{remark}

\begin{example}
In the case where $\lambda = (1)$, our intrinsic arrangement is just the braid arrangement whose complement space is the configuration space of $n$ points in $\C$, modulo simultaneous translation. That these spaces exhibit representation stability in their cohomology is essentially the prototypical and motivating example of representation stability, going back to the original papers \cite{CF,CEF}
\end{example}

\begin{definition}
Following the proof of Theorem \ref{compfg}, we will often use the notation
\[
\OS(\mathcal{A}) := H^\ast(\mathcal{M}(\mathcal{A})),
\]
for the cohomology ring (or \textbf{Orlik-Solomon algebra}) of the complementary variety. In the next section will will see these algebras play a prominent role. Importantly for us, one should recall that $\OS(\mathcal{A})$ is generated in its first degree, and that this first degree piece has a basis in bijection with the hyperplanes of $\mathcal{A}$.
\end{definition}

\section{$\FSop$-modules and Kazhdan-Lusztig polynomials}

\subsection{The technical setup}\label{fsmod}
In the following sections we tackle the problem of stability phenomena, or lack thereof, present in the equivariant Kazhdan-Lusztig polynomials of the intrinsic hyperplane arrangements $(S^{\lambda}_n,\mathcal{A}_{\lambda[n]})$. Perhaps somewhat surprisingly, these stability phenomena will not turn out to be related with $\FI$ or traditional representation stability at all. Instead, we will look at the category $\FSop$; the opposite category of finite sets and surjections. To begin, we recall the primary theorem on finitely generated modules over this category.

\begin{theorem}[\cite{sam,fs-braid,Tost1,Tost2}]\label{fsCon}
Let $V$ be a $\FSop$-module over $\Q$ that is a subquotient of a finitely generated $\FSop$ module that is generated in degree $\leq d$. Then for all $n \gg 0$:
\begin{enumerate}
    \item If $\lambda$ is a partition of $n$ such that $S^\lambda$ is a summand of $V_n$, then the number of rows of $\lambda$ is at most $d$.
    \item For any partition $\lambda$, the multiplicity of $\lambda[n]$ in $V_n$ agrees with a quasi-polynomial in $n$ of degree $\leq d-1$.
    \item There exist polynomials $\{p_i\}_{i \leq d}$ such that $\dim_\Q(V_n) = \sum_i p_i(n)\cdot i^n$.
\end{enumerate}
\end{theorem}

\begin{remark}\label{tostCor}
Just as with $\FI$-modules, there is a version of the final statement in the above theorem that encodes all of the character data of $V_n$ into a expression involving character polynomials and, unique to the $\FSop$ case, character exponentials. This statement is quite a bit more complex than the final statement of Theorem \ref{fgconsequence}, and we therefore opt to direct the interested reader to \cite{Tost2}, where it was originally proven.
\end{remark}

In previous works that have considered equivariant Kazhdan-Lusztig polynomials in the context of representation stability, the main tool that one has to prove stability results is a certain spectral sequence discovered by Proudfoot and Young in \cite{fs-braid}. To state this result, however, we need a particular construction coming from the theory of hyperplane arrangements.

\begin{definition}
Let $(V,\mathcal{A})$ be a complex hyperplane arrangement. A \textbf{flat} of the arrangement, $F$, is any linear subspace of $V$ that is realizable as an intersection of the hyperplanes of $\mathcal{A}$. That is to say, a flat of $\mathcal{A}$ is just any member of its intersection semi-lattice $\mathcal{L}(\mathcal{A})$. If $F$ is a flat, then the \textbf{contraction} by $F$ is the arrangement $(F,\mathcal{A}^F)$, where
\[
\mathcal{L}(\mathcal{A}^F) \cong \{x \in \mathcal{L}(\mathcal{A}) \mid x \subseteq F\}.
\]
On the other hand, the \textbf{restriction} to $F$ is the arrangement $(V/F, \mathcal{A}_F)$, where
\[
\mathcal{L}(\mathcal{A}_F) \cong \{x \in \mathcal{L}(\mathcal{A}) \mid x \supseteq F\}.
\]

The \textbf{matroidal corank} of a flat $F$ is equal to its dimension as a vector space over $\C$ $\crk(F) = \dim_{\C}(F)$.
\end{definition}

These definitions now allow us to describe the necessary spectral sequence.

\begin{theorem}[Proudfoot and Young, Theorems 3.1 and 3.3 \cite{fs-braid}]\label{spectral}
For any arrangement $(V,\mathcal{A})$ and positive integer $i$, there is a first quadrant homological spectral sequence $E(\mathcal{A},i)$ converging to $\IH_{2i}(X_{\mathcal{A}})$,
with $$E(\mathcal{A},i)^1_{p,q} = \bigoplus_{\crk G = p}\OS_{2i-p-q}(\mathcal{A}_G) \otimes \IH_{2(i-q)}(X_{\mathcal{A}^G}).$$
If $\varphi:\mathcal{A} \to \mathcal{A}^F$ is a contraction, there is a canonical map $\varphi^*:E(\mathcal{A}^F,i)\to E(\mathcal{A},i)$ of spectral sequences, composing in the expected way, and converging to a map $\IH_{2i}(X_{\mathcal{A}^F})\to \IH_{2i}(X_{\mathcal{A}})$.
The map $E(\mathcal{A}^F,i)^1_{p,q}\to E(\mathcal{A},i)^1_{p,q}$ kills the $G$-summand unless $G$ is contained in $F$.
In this case, the image of $G$
in $\mathcal{A}^F$ is a flat $G'$ of $\mathcal{A}^F$, and $(\mathcal{A}^F)^{G'}$ is canonically isomorphic to $\mathcal{A}^G$.  The map takes
the $G$-summand of $E(\mathcal{A},i)^1_{p,q}$ to the $G'$-summand of $E(\mathcal{A}^F,i)^1_{p,q}$ by 
the canonical map $\OS_{2i-p-q}(\mathcal{A}_G)\to\OS_{2i-p-q}(\mathcal{A}^{F}_G)$ tensored with the identity map on $\IH_{2(i-q)}(X_{\mathcal{A}^F})$.
\end{theorem}

Recall from Section \ref{sec:Comp} that the machinery of Gadish allowed us to conclude more than just formal representation stability for the cohomology groups of the complement of the arrangement. It allowed us to conclude that finite generation was true at the level of the arrangement itself. Combinatorially, this followed from studying the behavior of flats of the arrangement of constant \textbf{ matroidal rank}. That is, constant codimension within $S(\lambda)_n$. The above spectral sequence tells us that one can in principal prove similar statements for the coefficients of the equivariant Kazhdan-Lusztig polynomial, provided that one can show stability at the level of the flats of constant \textbf{matroidal corank}. 

We will discuss later why such stability appeared in the case of the braid arrangement, before arguing why it will likely \emph{not} appear in any longer hook diagram.

\subsection{Stable and unstable flats}

Let's return to our particular context. Fix a positive integer $l$, and write $\lambda$ for the partition $(1,\ldots,1)$ of $l$. Then the padded partition $\lambda[n]$ is a hook of height $l+1$. The intrinsic arrangements corresponding to hooks were considered in the originating work \cite{TVY}, as their combinatorics are some of the few that appear actually tractable. For instance, the $l = 1$ case is precisely the case of the braid arrangement. 

Hyperplanes in the arrangement $\mathcal{A}_n = \mathcal{A}_{\lambda[n]}$ are in bijection with subsets of $[n]$ of size $l+1$. Any surjection $f:[n] \twoheadrightarrow [m]$, is determined by the data of a partition of the set $[n]$, along with a bijection from $[m]$ to the parts of the partition. This data, in turn, determines a flat of $\mathcal{A}_n$ by taking the intersection of all hyperplanes whose corresponding set intersects at least one of the parts of the partition in such a way that the intersection has size at least two.  In other words, this flat is formed by the intersection of all hyperplanes $H_\alpha$ such that $|f(\alpha)| < l+1$. We write $F_f \in \mathcal{L}(\mathcal{A}_n)$ to denote this particular flat. These flats have a nice description as the following shows.

\begin{proposition}\label{specialFlats}
Let $f:[n] \rightarrow [m]$ be a surjection of sets. If we write $\Sn_f$ for the subgroup of $\Sn_n$ generated by permutations that only permute items within the fibres of $f$, then,
\[
F_f = S(\lambda)_n^{\Sn_f}.
\] 
In particular, if $f:[n] \rightarrow [n-1]$ is a surjection that maps $i$ and $j$ to the same element of $[n-1]$ then
\[
F_f = S(\lambda)^{(ij)}_n.
\]
\end{proposition}

\begin{proof}
For clarity of exposition we will only be proving the second statement on surjections $f:[n] \rightarrow [n-1]$ that identify $i$ and $j$. The general case is similar. 

To prove this fact, we switch perspectives from the hyperplanes composing $F_f$ to their normal vectors. Per \cite{TVY}, these normal vectors can be described in terms of the tabloid presentation $\{v_T\}_{T \in \mathcal{T}}$ of $S(\lambda)_n$ (see Remark \ref{Tabloid}). In particular, if $\alpha$ is a subset of $[n]$ of size $l+1$, then the normal vector of $H_{\alpha}$ is (up to a sign) equal to $v_T$ where $T$ assigns the numbers in $\alpha$ to the first column of $\lambda$. It follows from this that in the specific case of $F_f$, these tableaux will precisely be those for which both $i$ and $j$ lie in the first column. If one takes one such vector $v_T$, then by definition $(ij)$ acts by negation. In fact, those vectors $v_T$ containing both $i$ and $j$ in the first column precisely generate the sign representation of the subgroup generated by $(ij)$ within $S(\lambda)_n$. We can now conclude that the complementary space to this sign representation, i.e. $F_f$, must be the trivial representation of the subgroup generated by $(ij)$ within $S(\lambda)_n$.
\end{proof}

\begin{definition}
We will call all flats of the form $F_f$, for some surjection $f$, the \textbf{special flats} of $\mathcal{A}_n$.
\end{definition}

Our next proposition illustrates the importance of special flats.

\begin{proposition}\label{contract}
Let $f:[n] \twoheadrightarrow [m]$ denote a surjection of sets. Then $\mathcal{A}_n^{F_f}$ and $\mathcal{A}_m$ are isomorphic. This isomorphism assigns a hyperplane $H_\alpha$ of $\mathcal{A}_m$ to the hyperplane
\[
H_{g(\alpha)} \cap F_f,
\]
in $\mathcal{A}_n^{F_f}$, where $g:[m] \rightarrow [n]$ is any section of the surjection $f$
\end{proposition}

\begin{proof}

Write $\alpha_1,\ldots,\alpha_{m}$ for the (ordered) parts of the partition of $f$. Then $F_f$ is the intersection of all hyperplanes $H_{\gamma} \in \mathcal{A}_n$, such that $\gamma$ intersects at least one of the $\alpha_i$ in a set of size greater than 1.

Define a map $\phi_f:S^{\lambda}_m \rightarrow S^{\lambda}_n$ of vector spaces by setting,
\begin{align}
\phi_f = \frac{1}{n_f}\sum_{g} g_\ast, \label{section}
\end{align}
where the sum ranges over all injections of sets $g:[m] \hookrightarrow [n]$ that are sections of $f$, $g_\ast$ is the induced map of $g$ on the aforementioned $\FI$-module $S^{\lambda}_\bullet$, and $n_f$ is the number of sections of $f$. We claim that the map $\phi_f$ is injective. Indeed, viewing $S^{\lambda}_n$ as the wedge product $\bigwedge^l S^{(n-1,1)}$, $f$ itself induces an obvious map $f_\ast:S^{\lambda}_n \rightarrow S^{\lambda}_m$. That is, given our basis of wedges of differences of the canonical basis vectors of $\C^n$, $f$ will induce a map given by applying $f$ to the indices of these vectors to each term in the wedge product. It is not hard to check that this map is a splitting of $\phi_f$ - i.e. that $\phi_f$ is a section of it - and therefore $\phi_f$ must be injective.

To check that the image of $\phi_f$ is $F_f$, first observe that it must at least be contained in $F_f$. Indeed, if $H_{\alpha}$ is one of the hyperplanes intersected to obtain $F_f$, then by definition some part of the partition determined by $f$ must intersect $\alpha$ with size at least 2. Let's say $i,j$ are two indices that are contained in $\alpha$, while also being in a single part of $f$. Then we will show that the image of $\phi_f$ is contained in $(S^{\lambda}_n)^{(ij)}$. This shows one containment by Proposition \ref{specialFlats}. We see,
\[
(ij)\cdot \phi_f = \frac{1}{n_f}\sum_{g}(ij)\cdot g_\ast = \frac{1}{n_f}\sum_{g} ((ij) \circ g)_\ast = \frac{1}{n_f}\sum_{g} g_\ast = \phi_f.
\].

Note that in the second equality we have used the fact that the $\FI$-structure is functorial, whereas the third equality comes from the fact that $(ij)$ permutes two entries within a single fiber of the surjection $f$, and therefore $(ij) \circ g$ is once again a section of $f$.

Conversely, we must show that the map induced by $f$, when restricted to $F_f$, surjects onto $S(\lambda)_m$. Once again leaning on our description of the normal vectors, observe that for any tabloid corresponding to one such vector, there is a pair of elements $i,j$ in the same fibre of $f$ that both appear in its first column. In particular, the transposition $(ij)$ swapping these two elements negates the vector. However by definition $f \circ (ij) = f$ as functions of sets. By functoriality it follows that any vector negated by $(ij)$ must be the kernel of $f_\ast.$

Finally, let $H_\alpha$ be a hyperplane of $\mathcal{A}_m$ and let $g,g'$ be two sections of $f$. Then,
\[
H_{g(\alpha)} \cap F_f = H_{g'(\alpha)} \cap F_f.
\]
If $\sigma$ is the permutation for which $\sigma \circ g = g'$, then $\sigma \cdot H_{g(\alpha)} = H_{g'(\alpha)}$. By Proposition \ref{specialFlats} we know that the elements of $F_f$ are fixed by $\sigma$, completing the proof.

\end{proof}

\begin{definition}
    For each $n$, a flat of $\mathcal{A}_n$ is called \textbf{stable} if it is contained in some special flat. Otherwise, we refer to the flat as being \textbf{unstable}.
\end{definition}

\begin{example}
    Consider, once again, the case of the braid arrangement $\lambda = (1)$. In this setting, the lattice of flats is isomorphic to the partition lattice.

    Given a surjection $f:[n] \rightarrow [n-1]$, the special flat $F_f$ is precisely the hyperplane of the arrangement associated to the pair of elements that $f$ identifies. It follows from this that every flat of the arrangement is a stable flat so long as $n \geq 2$. This observation along with the spectral sequence of Theorem \ref{spectral} are the foundation of the main theorem of \cite{fs-braid}, that the equivariant Kazhdan-Lusztig coefficients form finitelty generated $\FSop$-modules.

    Observe, however, even for $\lambda = (1,1)$, it is no longer the case that the special flat $F_f$ corresponds to a hyperplane. We will more deeply consider the case of $\lambda = (1,1)$ in the next sections.
\end{example}

To conclude this section we provide the proof of Theorem \ref{bestwecando}.

\begin{proof}[Proof of Theorem \ref{bestwecando}]

To begin, notice that for $\lambda = (1,1,\ldots,1)$, the representation $S(\lambda)_{l+1}$ is the sign representation of $\Sn_{l+1}$. In particular, $S(\lambda)_{l+1}$ is 1-dimensional.

  If $f:[n] \rightarrow [l+1]$ is any surjection, then by Theorem \ref{contract} we find that the special flat $F_f$ must be 1-dimensional. Observe that if $f$ and $g$ are two surjections from $[n]$ to $[l+1]$ with different fibres, then $F_f \neq F_g$. Indeed, this follows from the classification in Proposition \ref{specialFlats}. We in particular conclude from this that the total number matroidal corank 1 flats in $\mathcal{A}_n$ is at least the Stirling number $S(n,l+1)$.

To conclude the proof, we note \cite{EPW} proves that for any arrangement $\mathcal{A}$, the first Kazhdan-Lusztig coefficient is equal to the number of matroidal corank 1 flats minus the number of matroidal rank 1 flats. In our case we know that the rank 1 flats are in bijection with subsets of $[n]$ of size $l$. In particular, this quantity is only polynomial in $n$. We therefore conclude that that first Kazhdan-Lusztig coefficient of $\mathcal{A}_n$ is bounded from below by a function that is asymptotically equal to $(l+1)^n$, as desired.
\end{proof}

\begin{remark} \label{expansion}
    In the above proof it was remarked that the linear coefficient of the Kazhdan-Lusztig polynomial is equal to the number of corank 1 flats minus the number of rank 1 flats. In fact, something stronger is true. There is a linear map from the formal vector space whose basis is in bijection with rank 1 flats to the formal vector space whose basis is in bijection with corank 1 flats that sends a rank 1 flat to the sum of all corank 1 flats that it \emph{does not} contain. The intersection homology group $IH_{2}(X_\mathcal{A})$, which can be thought of as the categorified version of this coefficient, is isomorphic to the cokernel of this map. This can be seen, with some work, from the spectral sequence of \ref{spectral}.

    By consequence, there is actually an equivariant version of Theorem \ref{bestwecando}, where one considers the $\FS^{op}$-module generated by the special flats coming from surjections $f:[n] \rightarrow [l+1]$. Theorem \ref{fsCon}, and more specifically the character behavior hinted at in Remark \ref{tostCor}, also implies exponential lower bounds on all character values.
\end{remark}

\subsection{The case of $\lambda = (1,1)$: statistics on stable and unstable flats}\label{notStable}
In this concluding section we present experimental evidence in the case of $\lambda = (1,1)$. We note that even though this is the smallest case beyond the well-understood braid matroid of $\lambda = (1)$, its combintorics already appears to be quite complicated. In what follows, we will consider the lines of the hyperplane arrangement $\mathcal{A}_n$. The code used in all that follows can be seen at \cite{FRYcode}.

To start, we found that the first interesting cases of $n = 4,5$ and $6$ could be computed in their entirety. Table \ref{realcount} displays what was discovered in these cases. Note that in this table, the \textbf{size} of a line refers to the maximum number of hyperplanes that one can intersect to obtain it. In other words, it is the set theoretic size of the flat in the associated matroid. In the ``Line Sizes" column of the table, a blue number indicates only unstable flats exist in that size, while a red number indicates only stable flats exist in that size
\begin{table}
\begin{center}
\begin{tabular}{c c c c c} 
 \hline
 $n$ & \# Lines & Line Sizes &  \% of Lines that are Unstable\\ [0.5ex] 
 \hline\hline
 4 & 6 & \textcolor{red}{2}  & 0\% \\ 
 \hline
 5 & 37 & \textcolor{blue}{5},\textcolor{blue}{6},\textcolor{red}{7} & 32.4\% \\
 \hline
 6 & 570 & \textcolor{blue}{9},\textcolor{blue}{10},\textcolor{red}{12},\textcolor{red}{14},\textcolor{red}{16} & 52.6\% \\ [1ex] 
 \hline
\end{tabular}
\caption{Exact counts for the number of lines in the arrangement $\mathcal{A}_n$ for $n = 4,5,6$.}\label{realcount}
\end{center}
\end{table}

There are a few features that one can immediately glean from Table \ref{realcount}. Firstly, in all cases beyond the $n = 4$ case there are unstable lines. In fact, all lines below a certain size are unstable. We will see in the cases of $n > 6$ that different sizes of lines can allow for both stable and unstable flats, however it is always the case that lines of smaller size tend to be unstable. Moreover, one observes that the total percentage of unstable lines is growing considerably with $n$.

We next outline the data we have accumulated for higher $n$. As stated above, for any $n$ at or above 7 it seems to become entirely infeasible to explicitly compute all lines. Instead, in these cases we switch to a more statistical approach.

To accumulate the data presented in the following table, we proceeded as follows for each $n \geq 7$:

\begin{itemize}
    \item Among the $\binom{n}{3}$ hyperplanes in $\mathcal{A}_n$, uniformly at random select a collection of $\binom{n-1}{2}-1$ to intersect. This is the minimum number of hyperplanes that one needs to intersect to create a line
    \item If this intersection happens to be a line, and that line has not yet been accounted for, add the collection of hyperplanes that created it to our current samples. If not, we ignore this particular collection of hyperplanes. Repeat this 300,000 times. 
    \item For each of the lines obtained in the second step, continue intersecting hyperplanes, making sure not to lose dimension, until a maximal intersection has been found describing the line. This maximal intersection description now allows one to easily remove any redundant lines.
    \item Check whether each of the remaining lines is stable by determining whether it is contained within the fixed locus of some transposition $(ij)$. To accomplish this, we consider the maximal intersections computed in the third step, and note the sets associated to the intersecting hyperplanes. A flat is stable if and only if there exists a pair $\{i,j\}$ for which all sets of the form $\{i,j,k\}$ appear.
\end{itemize}

The above procedure was accomplished in the open source software SAGE \cite{FRYcode}. Importantly, to make the linear algebra more feasible we wrote our code using the matroidal description of the intrinsic arrangements found in \cite{WWZ}.

Ultimately, our simulations produced the data found in the following table. In this case we continue to use the color coding of the previous table, while adding a violet color to indicate that both stable and unstable flats were found in that size.

\begin{table}
\begin{center}
\begin{tabular}{c c c c} 
 \hline
  $n$ & \# Lines found & Line sizes found &  $\frac{\text{\# Unstable Lines Found }}{\text{\# Lines found}}$\\ [0.5ex] 
 \hline\hline
 7 & 4040 & \textcolor{blue}{14-17},\textcolor{violet}{18},\textcolor{red}{19},\textcolor{violet}{20},\textcolor{violet}{21},\textcolor{red}{22-24},\textcolor{red}{26},\textcolor{red}{27},\textcolor{red}{30} & 23.6\% \\ 
 \hline
 8 & 5565 & \textcolor{blue}{23-26},\textcolor{violet}{27-32},\textcolor{red}{33},\textcolor{violet}{34},\textcolor{red}{35},\textcolor{red}{36},\textcolor{red}{38},\textcolor{red}{40},\textcolor{red}{42},\textcolor{red}{44},\textcolor{red}{46},\textcolor{red}{50} & 6.1\% \\
 \hline
 9 & 3004 & \textcolor{blue}{33},\textcolor{blue}{35},\textcolor{blue}{37-39},\textcolor{violet}{40-54},\textcolor{red}{55-64},\textcolor{red}{67-69},\textcolor{red}{72},\textcolor{red}{77} & 4.7\% \\
    \hline
 10 & 1195 & \textcolor{blue}{52},\textcolor{blue}{54},\textcolor{blue}{55},\textcolor{blue}{57},\textcolor{blue}{58},\textcolor{blue}{60-62},\textcolor{red}{63},\textcolor{violet}{64-67},\textcolor{blue}{68},\textcolor{violet}{69-76},\textcolor{red}{77-86}, & 3.8\%\\ & &\textcolor{red}{88},\textcolor{red}{90},\textcolor{red}{92},\textcolor{red}{94},\textcolor{red}{96},\textcolor{red}{100},\textcolor{red}{102},\textcolor{red}{106},\textcolor{red}{112}\\
    \hline
 11 & 394 & \textcolor{blue}{75},\textcolor{blue}{78},\textcolor{blue}{79},\textcolor{blue}{87},\textcolor{blue}{89},\textcolor{blue}{90},\textcolor{red}{92-94},\textcolor{violet}{95},\textcolor{blue}{97},\textcolor{violet}{98},\textcolor{violet}{99},\textcolor{blue}{100},& 3\% \\ & & \textcolor{violet}{101-103},\textcolor{red}{104-130},\textcolor{red}{132-135},\textcolor{red}{137},\textcolor{red}{140-142},& \text{}\\ & & \textcolor{red}{144},\textcolor{red}{149},\textcolor{red}{156}\\ [1ex] 
 \hline
\end{tabular}
\caption{Compilation of the various statistics found using our sampling procedure.}\label{stabtab}
\end{center}
\end{table}

This table reveals a number of interesting things about the lines in the arrangement $\mathcal{A}_n$. While the proportion of unstable flats grows from $n = 4$ to $n = 6$, our method finds proportionally fewer and fewer unstable lines for larger $n$. It is entirely possible this is an artifact of how our code samples lines. Recall that we start by uniformly at random choosing $\binom{n-1}{2} - 1$ hyperplanes to intersect. A line of very large size will therefore tend to be found through this method more often than a line of very small size. As we have seen, the larger flats are generally stable while the smaller flats are generally unstable. This bias might be affecting our data. On the other hand, without having the exact numbers on how many lines exist at each size, it is hard to say how impactful this bias really is. This reasoning also might suggest the reason why our code failed to find lines of minimum size, i.e. $\binom{n-1}{2}-1$, after $n = 7$. In the next section we apply statistical means to approximate the number of lines in the $n = 7$ case in each size. These approximations agree with the idea that there are far more lines in small sizes that our sampling procedure is biased against finding.

\subsection{The case of $\lambda = (1,1)$: approximate counting}

We would like to know - even approximately - how many lines there are in the matroid.  That is, to what extent have the flats that we've found above exhaust the ones which are available?

The question is compounded by the fact that the flat-finding algorithm above does not produce flats with equal probability.  However, when we condition on the size of the flats, the probabilities *are* equal.  That is, the flat-finding algorithm has the following property.  Let $E$ be the event that the algorithm successfully finds a flat of size $k$.  Then, conditioned on $E$, the algorithm is sampling from the uniform distribution on $m$-flats.

This observation allows us to perform an "approximate counting" task, as follows:  We first obtain a collection of flats, by running our algorithm with $n=7$ 5,000,000 times.  We sort the flats by size, yielding random sets $\mathcal{F}_m$ for various flat sizes $m$.  We then do this same task again, yielding a second family of random sets $\mathcal{F}_m'$.  For each $m$, we approximate the number of lines in the matroid having that flat size as $$\#\{m-\text{flats}\} \approx \frac{|\mathcal{F}_m|\cdot|\mathcal{F}'_m|}{|\mathcal{F}_m \cap \mathcal{F}_m'|}.$$ We present the results of one run of this computation below -  If any of $|\mathcal{F}_m|, |\mathcal{F}_m'|, |\mathcal{F}_m \cap \mathcal{F}_m'|$ are zero, then it doesn't make any sense to do the computation - the best we can do is to report the size of whichever of these sets aren't zero, if any. While there is also considerable variation in the count sizes, due to the random sampling, we believe we can be confident in the rough order of magnitude.

\begin{table}
\begin{center}
\begin{tabular}{c c c c c} 
 \hline
  Line Size & First Round Count & Second Round Count & Overlap Count & Approx. count\\ [0.5ex] 
 \hline\hline
 14 & 6 & 11 & 0 & N/A \\ 
 \hline
 15 & 162 & 173 & 1 & 28,026 \\
 \hline
 16 & 406 & 425 & 9  & 19,172 \\
 \hline
 17 & 1,566 & 1,514 & 187 & 12,679\\
 \hline
 18 & 3,195 & 3,307 & 1,245 & 8,487\\ 
 \hline
 19 & 2,028 & 2,058 & 1,651 & 2,528\\
 \hline
 20 & 2,484 & 2,481 & 2,448 & 2,517\\
 \hline
 21 & 1,620 & 1,620 & 1,620 & 1,620\\
 \hline
 22 & 630 & 630 & 630  & 630\\
 \hline
 23 & 735 & 735 & 735 & 735\\
 \hline
 24 & 420 & 420 & 420 & 420\\
 \hline
 26 & 70 & 70 & 70 & 70\\
 \hline
 27 & 105 & 105 & 105 & 105\\
 \hline
 30 & 21 & 21 & 21 & 21
\end{tabular}
\caption{Approximations of the total number of lines in each size for $n = 7$ obtained using our sampling procedure.}\label{axct}
\end{center}
\end{table}

Looking at Table \ref{axct}, one immediately sees that we very likely have an exact count for the number of lines of sizes 21 and higher. This is the result of how much larger lines are favored by our sampling procedure. The data in the above table also suggests that these larger lines are far fewer in number than the smaller lines. The previous sections have already implied that smaller lines are far more likely to be unstable, leading us to conjecture that there will always be far more unstable lines than stable ones. In fact, if we look back to the $n = 7$ row of Table \ref{stabtab}, and make the simplifying, albeit slightly erroneous, assumption that all lines of sizes 19 and higher are stable, while those of size 18 and lower are unstable, the above counts suggest that at least 90\% of all lines are unstable. Being that we lack even an approximate count for the lines of size 14, this number is bound to be even closer to 100\%. Once again, this points to a kind of representation stability result similar to \cite{fs-braid} being extremely unlikely outside of the braid arrangement. There is perhaps still hope, however, that stability is achieved via the action of a category different than $\FSop$

\bibliography{./symplectic}
\bibliographystyle{amsalpha}

\end{document}